\documentclass[leqno,12pt]{article}
\usepackage{amssymb, amsmath, amsthm, amscd}
\usepackage[all]{xy}

\theoremstyle{plain}
\newtheorem{theorem}{Theorem}[section]
\newtheorem{proposition}[theorem]{Proposition}
\newtheorem{lemma}[theorem]{Lemma}
\newtheorem{corollary}[theorem]{Corollary}

\theoremstyle{remark}
\newtheorem*{remark}{\it Remark\/}

\theoremstyle{definition}
\newtheorem{definition}[theorem]{Definition}

\def\bn{\bigskip\noindent}
\def\mn{\medskip\noindent}
\def\sn{\smallskip\noindent}
\def\a{\alpha}
\def\bi{{b_i}}
\def\boi{{b_{0i}}}
\def\C{\mathbb{C}}
\def\chibar{\bar{\chi}}
\def\ci{{c_i}}
\def\Cl{{\mathbb{C}_\ell}}
\def\congss{\equiv_\ss}
\def\ebar{\bar{\epsilon}}
\def\ebari{{\bar\varepsilon_i}}
\def\ebarl{\ebar_\lambda}
\def\ee{\varepsilon}

\def\ei{{\varepsilon_i}}
\def\eibar{\bar{\varepsilon}_i}
\def\El{{E_\lambda}}
\def\End{\mathrm{End}}
\def\equivss{\equiv_\ss}
\def\et{{\mathrm{et}}}

\def\F{\mathbb{F}}
\def\Fl{\F_\ell}
\def\fl{{f_\l}}
\def\Flh{\F_{\ell^h}}
\def\Flhi{\F_{\ell^{h_i}}}
\def\Frob{\mathrm{Frob}}
\def\Frobv{{\Frob_v}}
\def\Frobvp{{\Frob_{v'}}}

\def\Gal{\mathrm{Gal}}
\def\GK{G_K}
\def\Gkv{{G_{\kv}}}

\def\GL{\mathrm{GL}}
\def\Gl{G_\ell}
\def\GQ{G_\Q}
\def\Gq{G_q}
\def\Gu{G_u}
\def\Gv{G_v}

\def\HT{\mathrm{HT}}
\def\HTl{\HT_\ell}
\def\HTu{\HT_u}
\def\Hetr{H_\et^r}
\def\Hets{H_\et^s}
\def\i{\iota}
\def\Im{\mathrm{Im}}
\def\incl{\hookrightarrow}
\def\Ind{\mathrm{Ind}}
\def\IndKQ{\Ind_{\GK}^{\GQ}}
\def\IndKuQl{\Ind_{\Gu}^{\Gl}}
\def\IndKvQq{\Ind_{\Gv}^{\Gq}}
\def\Iu{{I_u}}
\def\Iutame{{I_u^\tame}}
\def\Iuptame{{I_{u'}^\tame}}
\def\Iv{{I_v}}
\def\Ivp{{I_{v'}}}
\def\Ivtame{I_v^\tame}
\def\kappabar{\bar{\kappa}}
\def\Kbar{\bar{K}}
\def\kl{{k_\l}}
\def\Ku{{K_u}}
\def\Kup{{K'_{u'}}}
\def\Kv{{K_v}}
\def\kv{{k_v}}
\def\Kvbar{\Kbar_v}
\def\Kvp{{K'_{v'}}}

\def\l{\lambda}
\def\O{\mathcal{O}}
\def\OE{{\O_E}}
\def\OEl{{\O_\El}}

\def\pmodl{\pmod{\l}}
\def\Q{\mathbb{Q}}
\def\Qbar{\bar{\Q}}
\def\Ql{{\Q_\ell}}
\def\Qlbar{{\bar{\Q}_\ell}}
\def\Qq{{\Q_q}}
\def\qv{q_v}
\def\R{\mathbb{R}}

\def\RepGEln{\mathrm{Rep}^{\mathrm{(G)}}_{E,\l,n}}
\def\RepGElnKubev{\RepGEln(K;u,b,e,v)}
\def\RepGElnKubevp{\RepGEln(K;u,b,e,v)'}

\def\rfl{\rho_{f,\lambda}}
\def\rflbar{\bar{\rho}_{f,\lambda}}
\def\simeqss{\simeq_\ss}
\def\SkNe{S_k(N,\epsilon)}
\def\ss{\mathrm{ss}}
\def\sv{{\sigma_v}}
\def\tame{\mathrm{tame}}
\def\Tbar{\bar{T}}

\def\TI{\mathrm{TI}}
\def\TIu{\TI_u}
\def\Tl{T_\ell}
\def\Tr{\mathrm{Tr}}
\def\Vbar{\bar{V}}
\def\Vbarss{\Vbar^\ss}
\def\VbarssXr{\Vbar_X^{r,\ss}}
\def\VbarXr{\Vbar_X^r}

\def\Vfl{V_{f,\lambda}}
\def\Vss{{V^\ss}}
\def\VXr{V_X^r}
\def\W{\mathrm{W}}
\def\Wbar{\bar{W}}
\def\Wq{\mathrm{W}_q}
\def\Wv{\mathrm{W}_v}
\def\XKbar{X_{\Kbar}}
\def\XKvbar{X_{\Kvbar}}
\def\YKbar{Y_{\Kbar}}
\def\Z{\mathbb{Z}}
\def\Zbar{\bar{\Z}}
\def\Zl{{\Z_\ell}}
\begin{document}

\begin{center}
{\Huge On congruences of Galois representations
of number fields}

\vskip 1cm
{\large Yoshiyasu Ozeki and Yuichiro Taguchi}
\end{center}


\vskip 1.5cm
\begin{abstract}
We give a criterion for two $\ell$-adic Galois representations of an 
algebraic number field to be isomorphic when restricted to a 
decomposition group, in terms of the global representations mod $\ell$. 
This is applied to prove a generalization of a conjecture of 
Rasmussen-Tamagawa \cite{Ras-Tama} under a semistablity condition, 
extending some results \cite{Ozeki} of one of the authors. 
It is also applied to prove a congruence result on 
the Fourier coefficients of modular forms.
\end{abstract}



\section{Introduction}\label{sect:intro}
Let  $K$  be an algebraic number field (:= finite extension of  $\Q$) 
and let 
$\GK=\Gal(\Kbar/K)$  denote its absolute Galois group, where  
$\Kbar$  is a fixed algebraic closure of  $K$.  
Choosing an extension of  $v$  to  $\Kbar$, we denote by  
$\Gv$  (resp.\ $\Iv$) the decomposition (resp.\ inertia) group of  
$v$  in  $\GK$. 
Let  $E$  be another algebraic number field, 
$\l$  a finite place of  $E$  of residue characteristic  $\ell$, and  
$\El$  the completion of  $E$  at  $\l$.  We denote by  
$\OE$  and  $\OEl$  the integer rings of  
  $E$  and   $\El$, respectively. 
Let  $\fl$  denotes the absolute residue degree of  $\l$.
We identify any finite place  $v$  of an 
algebraic number field with the corresponding prime ideal, and 
denote its residue field by  $\kv$  and put  $\qv:=\#\kv$. 
Throughout the paper, we fix  $K$, $E$, and a finite place  
$v$  of  $K$, and let the finite place  $\l$  of  $E$  
of residue characteristic  $\ell$  vary. 
We denote by  $\ell$  the residue characteristic of  $\l$, and 
assume  $v\nmid\ell$, while 
$u$  will denote another finite place of  $K$  lying above  $\ell$. 
All representations of Galois groups denoted  $V$  are either  
$\Ql$- or $\El$-linear of finite dimension, 
and assumed to be continuous with respect to the natural topologies. 
Their ``reductions'' will be denoted by  $\Vbar$. 

In the following, 
$n$  and  $e$  are fixed integers $\geq 1$  and  
$e$  is assumed to be divisible by the absolute ramification index  
$e(\Ku/\Ql)$  of  $\Ku/\Ql$.
For  $K,u,v,E,\l,n,e$  as above and a real number $b$, let  
$\RepGElnKubev$  denote the set of 
$n$-dimensional $\El$-linear representations  $V$  of  $\GK$  
which have the following properties:

\sn
-- $V$  is semistable at  $v$  (in the sense that the action of the 
inertia is unipotent (including the case where it is trivial)),
\\
-- $V$  is $E$-integral at  $v$  in the sense of Definition \ref{def:E-int}, 
\\
-- $V$  becomes semistable (in the sense of Fontaine \cite{Fontaine}) 
over a finite extension  $\Kup$  of  $\Ku$  
whose absolute ramification index  $e(\Kup/\Ql)$  divides  $e$, 
\\
-- $V$  has Hodge-Tate weights  $\subset[0,b]$  at  $u$, and 
\\
-- $V$  is of type (G) in the sense of Definition \ref{def:(G)},
  
\smallskip
Our first main result is:

\begin{theorem}\label{thm:main}
For any  $K,E,n,b,v$  as above, there exists a constant  
$C=C([E:\Q],n,b,e,\qv)$  such that the following holds: 
For any prime number  $\ell>C$, 
any places  $u$  of  $K$  and  $\l$  of  $E$  
both lying above  $\ell$,  and any representations  
$V \in\RepGElnKubev$  and  
$V'\in\RepGEln(K;u,(\ell-2)/e^2,e,v)$, 
if one has  
$V\congss V'\pmodl$  both as 
$\Gu$-representations and  
$\Gv$-representations, then one has  
$V\simeqss V'$  as $\Gv$-representations.
[In particular, if   
$V\congss V'\pmodl$  as $\GK$-representations, then  
$V\simeqss V'$  as $\Gv$-representations.]

The constant  $C$  can be taken explicitly to be
$$
    C\ :=\ \max\{e^2b+1,\left(2{n\choose[n/2]}\qv^{nb}\right)^{[E:\Q]/\fl}\},
$$
where  
$[x]$  denotes the largest integer not exceeding  $x$. 
\end{theorem}

Here, 
the meaning of the notations 
$\congss$  and  $\simeqss$  is as follows: 
we say 
$V\congss V'\pmodl$  as $\Gv$-representations 
if  $T$  and  $T'$  are  $\Gv$-stable $\OEl$-lattices in  
$V$  and  $V'$, respectively, and the 
semisimplifications  $(T/\l T)^\ss$  and   $(T'/\l T')^\ss$ 
are isomorphic as  $\kl$-linear representations of  $\Gv$ 
(this definition does not depend on the choice of the lattices). 
We say also  
$V\simeqss V'$  as $\Gv$-representations if 
their semisimplifications are isomorphic as 
$\El$-linear representations of  $\Gv$. 

To state a variant of this theorem, let  
$\RepGElnKubevp$  be the set of 
$n$-dimensional $\El$-linear representations  $V$  of  $\GK$  
which have the following properties:

\sn
-- $V$  is $E$-integral at  $v$,  
\\
-- $V$  becomes semistable 
over a finite extension  $\Kup$  of  $\Ku$  
whose absolute ramification index  $e(\Kup/\Ql)$  divides  $e$, 
\\
-- $V$  has Hodge-Tate weights  $\subset[0,b]$  at  $u$, and 
\\
-- $V$  is of type (G).

\sn
Thus 
$\RepGElnKubevp$  contains  
$\RepGElnKubev$, and the difference is that 
the elements  $V$  of the former are not assumed to be semistable at  $v$. 
Let  $\Wv(V)$  denote the multi-set of 
Weil weights of  $V$  (Def.\ \ref{def:Weil})
considered as a $\Ql$-linear representation of  $\Gv$. 

\begin{theorem}\label{thm:main2}
For  $K,E,n,b,v$  as above, the following holds
with the same constant 
$C=C([E:\Q],n,b,e,\qv)$  as in Theorem \ref{thm:main}: 
For any prime number  $\ell>C$, 
any places  $u$  of  $K$  and  $\l$  of  $E$  
both lying above  $\ell$, and any representations  
$V \in\RepGElnKubevp$  and  
$V'\in\RepGEln(K;u,(\ell-2)/e^2,v)'$, 
if one has  
$V\congss V'\pmodl$  both as 
$\Gu$-representations and  
$\Gv$-representations, then one has  
$\Wv(V)=\Wv(V')$.
[In particular, if 
$V\congss V'\pmodl$  as $\GK$-representations, then   
$\Wv(V)=\Wv(V')$.]
\end{theorem}

\begin{remark}
If we consider representations of type (W) at {\it all} places  $v|q$  
for a fixed prime number  $q$
and of Hodge-Tate type at {\it all} places  $u|\ell$,
we can prove versions of Theorems \ref{thm:main} and \ref{thm:main2} 
without assuming ``type (G)'' but with 
a larger constant
$$
    C'\ :=\ \max\{e^2b+1,\left(2{n\choose[n/2]}q^{nb[K:\Q]/[\Kv:\Qq]}\right)^{[E:\Q]/\fl}\}.
$$
The proofs are basically the same as in the case of type (G) 
but use Proposition \ref{prop:(G)} instead of the equality (G) 
in Definition \ref{def:(G)}.
\end{remark}

The constant  $C=C([E:\Q],n,b,e,\qv)$  above depends on
the coefficient field  $E$. 
By working mod $\ell$ rather than mod $\lambda$, however, 
we can suppress this dependence on  $E$  as follows: 

\begin{theorem}\label{thm:main3}
For any  $K,E,n,b,v$  as above, there exists a constant  
$\tilde{C}=\tilde{C}(n,b,e,\qv)$  such that the following holds: 
For any prime number  $\ell>C$, 
any places  $u$  of  $K$  and  $\l$  of  $E$  
both lying above  $\ell$,  and any representations  
$V \in\RepGElnKubev$  and  
$V'\in\RepGEln(K;u,(\ell-2)/e^2,e,v)$, 
if one has  
$V\congss V'\pmodl$  as $\Gu$-representations and 
$\det(T-\Frobv |V)\equiv 
 \det(T-\Frobv |V')\pmod{\ell\OE}$, 
then one has  
$V\simeqss V'$  as $\Gv$-representations. 
[In particular, if   
$V\congss V'\pmod{\ell}$  as $\GK$-representations, then  
$V\simeqss V'$  as $\Gv'$-representations.]

The constant  $\tilde{C}$  can be taken explicitly to be
$$
    \tilde{C}\ :=\ \max\{e^2b+1,2{n\choose[n/2]}\qv^{nb}\}.
$$
\end{theorem}

\medskip
After recalling some notions and results 
on Galois representations in Section 2, 
we give proofs of the above theorems in Section 3 and 
several corollaries of Theorem \ref{thm:main2} in Section 4. 
In Section 5, we apply Theorem \ref{thm:main3} with  $E$  a Hecke field  
to prove a congruence result on the Fourier coefficients of modular forms 
of various levels, 
where the ``independence of  $E$''  in the theorem  
plays a significant role. 

\mn
{\it Acknowledgments.} 
The second-named author thanks Eknath Ghate for his invitation to TIFR, Mumbai, 
and his interest in this work, 
which motivated us to write down the results;  
both the authors are grateful to him for his useful comments on 
the first version of this paper.
The authors thank Tetsushi Ito and Yoichi Mieda for their 
useful information on $\ell$-adic \'etale cohomology. 
This work is supported in part 
by JSPS Fellowships for Young Scientists 
and JSPS KAKENHI 22540024.

\section{Weights}\label{sect:wt}

\noindent
{\it 2.1. Weil weights.} 
Let  $V$  be a $\Ql$-linear representation of  $\Gv$. 
Choose a lift  $\sv\in\Gv$  of the $\qv$-th power Frobenius  
$\Frobv\in\Gkv$  and let  
$P(T)=\det(T-\sv|V)$  be the characteristic polynomial of  
$\sv$  acting on  $V$.  
Recall that an algebraic integer  $\a$  is said to be 
a $q$-Weil integer of weight  $w$  if  
$|\i(\a)|=q^{w/2}$  for any field embedding  
$\Qbar\incl\C$, where  
$|\cdot|$  denotes the absolute value of  $\C$.

\begin{definition}
\label{def:Weil}
We say that  $V$  is of {\it type (W)} at  $v$  
if all the roots of  $P(T)$  are  $\qv$-Weil integers. 
If this is the case, we call the weights of the roots of  $P(T)$  
the {\it Weil weights} of  $V$  at  $v$, and denote by  
$\Wv(V)$  the multi-set consisting of them. 
\end{definition}

This definition does not depend on the choice of 
the Frobenius lift  $\sv$. 
Also, the multi-set  $\Wv(V)$  is unchanged by 
a finite extension of the base field  $\Kv$.

Now suppose  $V$  is an $\El$-linear representation of  $\Gv$. 
The action of the inertia subgroup  $\Iv$  on  $V$  is quasi-unipotent 
(\cite{Serre-Tate}, Appendix); thus there exists a finite extension  
$\Kvp/\Kv$  such that the inertia subgroup  $\Ivp$  for  
$\Kvp$  acts unipotently on  $V$  (or equivalently, trivially on the 
semisimplification  $\Vss$  as an $\El[\Gv]$-module). 
Hence we can consider the characteristic polynomial  
$P'(T)=\det(T-\Frobvp|\Vss)$  of the Frobenius  $\Frobvp$  at  $v'$  
acting on the $\El$-vector space  $\Vss$. 
(Note that the characteristic polynomial taken with  
$\Vss$  viewed as a $\Ql$-vector space is the product of 
the ``conjugates'' of this  $P'(T)$.)

\begin{definition}
\label{def:E-int}
An $\El$-linear representation $V$ of $\Gv$ is said to be 
{\it $E$-integral} at  $v$  if, 
for any finite extension  $\Kvp/\Kv$  for which the 
inertia action on  $V$  is unipotent, 
the characteristic polynomial  $P'(T)$  
defined as above has coefficients in  $\OE$.  
\end{definition}

Note that an $E$-integral representation of type (W) at $v$  
has Weil weights $\geq 0$ at $v$. 

For example, 
if  $X$  is a proper smooth variety over  $\Kv$, then  
the $\Ql$-linear dual  $V=\Hetr(\XKvbar,\Ql)^*$ 
of the $r$-th $\ell$-adic \'etale cohomology group  
of  $\XKvbar:=X\otimes_{\Kv}\Kvbar$  is conjectured to be 
$\Q$-integral (cf.\ \cite{Serre:FL}, C$_4$). 
This conjecture is known to be true under the assumption 
of the existence of the K\"unneth projector 
(\cite{Saito}, Cor.\ 0.6 (1)).

We note here that, by the next lemma, 
there are totally ramified extensions among 
the finite extensions  $\Kvp/\Kv$  as above 
(so that, when we want to compare the characteristic polynomials  
$P'(T)$  for different  $V$'s, we can use a  $\Kvp$  with 
residue degree  $f=1$): 

\begin{lemma}
\label{lem:elimination}
If  $L/\Kv$  is a finite Galois extension, then 
there exists a totally ramified subextension  $L'/\Kv$  of  $L/\Kv$  such that  
$L=L'L_0$, where  
$L_0$  is the maximal unramified subextension of  $L/\Kv$. 
\end{lemma}

\begin{proof} 
If  $L/\Kv$  is abelian, this is a consequence of 
local class field theory. 
Suppose  $L/\Kv$  is non-abelian. 
We proceed by induction on the extension degree  $[L:\Kv]$.  
Let  $\sigma$  be a lift in  $G:=\Gal(L/\Kv)$  of the 
Frobenius in  $\Gal(L_0/\Kv)$, and set  
$H:=\langle\sigma\rangle$. 
Then we have  
$H\subsetneqq G$, and the extension  
$L^H/\Kv$  is a non-trivial totally ramified subextension 
of  $L/\Kv$. 
Repeating this process with  $L/\Kv$  replaced by  $L/L^H$, 
we are reduced to the case of abelian  $L/\Kv$.
\end{proof}

\bn
{\it 2.2. Hodge-Tate weights.} 
Recall that  $u$  is a finite place of  $K$  lying above  $\ell$.
A $\Ql$-linear representation  $V$  of  $\Gu$  
is said (cf.\ \cite{Fontaine}
) to be of 
{\it Hodge-Tate type} of {\it Hodge-Tate weights}  
$h_1,...,h_n$, where  $n=\dim_\Ql(V)$  and  $h_i$  are integers, 
if one has  $V\otimes_{\Ql}\Cl\simeq\Cl(h_1)\oplus\cdots\oplus\Cl(h_n)$  
as a $\Cl$-semilinear $\Gu$-representation, where  
$\Cl(h)$  denotes the  $h$-th Tate twist of the completion  $\Cl$  of 
a fixed algebraic closure  $\Qlbar$  of  $\Ql$. 
If this is the case,  let 
$\HTu(V)$  denote the multi-set of Hodge-Tate weights of  $V$.  
Note that  $\HTu(V)$  is unchanged by a finite extension 
of the base field  $\Ku$.

\bn
{\it 2.3. Tame inertia weights.}
Let  $\Iutame$  the tame inertia group of  $K$  at  $u$  
(= the quotient of the inertia group  $\Iu$  at  $u$  by its 
maximal pro-$\ell$ subgroup). 
A character  $\varphi:\Iutame\to\Flh^\times$  can be written in the form  
$\varphi=\psi_1^{t_1}\cdots\psi_h^{t_h}$,  where  
$\psi_i$  are the fundamental characters of level  $h$  (\cite{Serre:PG}, \S 1.7) 
and  $0\leq t_i\leq\ell-1$. 
Then we set  
$\TIu(\varphi):=\{t_1/e,...,t_h/e\}$ (as a multi-set), where  
$e=e(\Ku/\Ql)$  is the ramification index of  $K/\Q$  at  $u$. 
Note that, by \S 1.4 of \cite{Serre:PG}, 
$\TIu(\varphi)$  is unchanged by a ``moderately'' ramified extension of  $\Ku$; 
precisely speaking, if  $\Kup/\Ku$  is a finite extension of 
ramification index  $e(\Kup/\Ku)<(\ell-1)/\max\{t_j|\ 1\leq j\leq h\}$, 
then we have  
$\TI_{u'}(\varphi|_{I_{u'}^\tame})=\TIu(\varphi)$.

Let  $V$  be a $\Ql$-linear representation of  $\Gu$, and 
$T$  a $\Gu$-stable $\Zl$-lattice of  $V$.
Set  $\Tbar:=T/\ell T$. 
Then its semisimplification  
$\Tbar^\ss$  (as an $\Fl[\Gu]$-module) is tamely ramified 
(note that its isomorphism class does not depend on the choice of  $T$), 
and the action of the tame inertia group  $\Ivtame$  is described by 
a sum of characters  
$\varphi_i:\Ivtame\to\Flhi^\times$. 
Then we define  
$\TIu(V)$  (as a multi-set) to be the union of the  
$\TIu(\varphi_i)$  for all  $i$.

\bn
{\it 2.4. Weights of geometric Galois representations.} 
Let  $V$  be a $\Ql$-linear representation of  $\GK$. 
For any multi-set  $X$, we write
$$
    \Sigma(X)\ :=\ \sum_{x\in X}x,
$$
whenever the sum on the right-hand side has a meaning.

\begin{definition}
\label{def:(G)}
We say that  $V$  is of {\it type (G)} if it is 
of type (W) at  $v$, of Hodge-Tate type at  $u$, and one has
\begin{equation}
\tag{G}
  \Sigma(\Wv(V))\ =\ 2\Sigma(\HTu(V)).
\end{equation}
If this is the case, we denote this value by  $w(V)$  and  call it the 
{\it total weight} of  $V$. 
\end{definition}

Note that   $\Sigma(\Wv(V))$  and  $\Sigma(\HTu(V))$  are 
respectively the Weil and Hodge-Tate weights of  $\det_\Ql(V)$. 

Typical examples of  $V$  of type (G) include the Tate twists 
$\Ql(r)$  for  $r\in\Z$  and their twists by characters of finite order; 
their total weights are  $2r$. 

A priori, the notion of type (G) depends on the places  
$v\nmid\ell$  and  $u\mid\ell$ 
(so it should be called, say, type (G$_{u,v}$)), but in practice 
(i.e., in case  $V$  comes from algebraic geometry), 
it should be independent of the places. 
The proof of the following proposition, 
which is modeled on the proof of Lemma 2.1 of \cite{Stiefel-Whitney}, 
has been communicated to us by Yoichi Mieda, 
to whom we are grateful:

\begin{proposition}
\label{prop:(G)}
Let  $X$  be a proper smooth variety over  $K$. 
Let  $V=\Hetr(\XKbar,\Ql)^*$  be the $\Ql$-linear dual
of the $r$-th $\ell$-adic \'etale cohomology group of  
$\XKbar:=X\otimes_K\Kbar$, and put  
$n=\dim_\Ql(V)$. 
Then we have:

\sn
(i) 
$\det(V)$  is isomorphic to the twist of  
$\Ql(nr/2)$  by a character  $\ee$  of order at most $2$. 
If  $r$  is odd, then  $\varepsilon=1$. 

\sn
(ii)
$V$  is of type (G) with respect to any finite places  
$u\mid\ell$  and  $v\nmid\ell$  of  $K$.
\end{proposition}

Note that, in (i), the Betti number  $n$  is even if  $r$  is odd 
by, say, the Hodge symmetry. 

\begin{proof}
(ii) follows from (i) immediately. To show 
(i), consider the character  
$\ee:\GK\to\Ql^\times$  
defined by  
$\det(V)(-nr/2)$, where  $(-nr/2)$  denotes the 
$(-nr/2)$-th Tate twist. 
If  $v$  is a finite place of  $K$  where  $X$  has good reduction, 
then by \cite{WeilII}  
$V$  is  $\Q$-integral and has all Weil weights equal to  $r$. 
Hence  $\ee(\Frobv)$  is a Weil integer in  $\Q$  of weight  $0$, i.e.,\ 
a unit of  $\Z$. 
Since  $\Frobv$'s for such  $v$'s are dense in  $\GK$, 
we see that  $\ee$  takes values in  $\Z^\times$. 
The second statement of (i) follows from 
Corollary 3.3.5 of \cite{Suh}.
\end{proof}

In some cases, 
we can expect the total weight  $w(V)$  to be equal also to  $2\Sigma(\TIu(V))$:

\begin{proposition}
\label{prop:tame inertia}
Let  $V$  be a $\Ql$-linear semistable representation of  $\Gu$  
with $\HTu(V)\subset[0,b]$.  
If  $e(\Ku/\Ql)b<\ell-1$, then we have:

\sn
(i) (\cite{Caruso}, Thms.\ 1.0.3 and 1.0.5) 
$\TIu(V)\subset[0,b]$.

\sn
(ii) (\cite{Caruso-Savitt}, Thm.\ 1)
$\Sigma(\HTu(V))=\Sigma(\TIu(V))$.
\end{proposition}

The equality (G) holds in general if  $K=\Q$:
\begin{lemma}\label{lem:Q}
Let  $q$  be a prime number  $\not=\ell$. 
If  $V$  is a $\Ql$-linear representation of  $\GQ$  
which is of type (W) at  $q$  and of Hodge-Tate type at  $\ell$, 
then  $V$  is of type (G).
\end{lemma}

\begin{proof}
By taking the determinant, we are reduced to the case 
$\dim_\Ql(V)=1$. Then  $V$  is geometric (in the sense of 
Fontaine-Mazur \cite{Fontaine-Mazur} 
(note that a one-dimensional $\Ql$-representation
is de Rham if and only if it is Hodge-Tate) and hence is a 
twist by a finite character of  $\Ql(r)$  
for some integer  $r$. 
Thus (G) holds for  $V$. 
\end{proof}

If  $K\not=\Q$, the equality (G) may not hold 
even for a geometric representation. 
For example, let  
$K$  be an imaginary quadratic field, 
$E$  an elliptic curve over  $K$  such that 
$\End_K(E)\otimes_\Z\Q\simeq K$, and  
$\ell$  a prime number which splits in  $K$  as  
$\ell=\l\l'$. Let  
$V$  be a  one dimensional $\GK$-subrepresentation of 
the $\ell$-adic Tate module $\Tl(E)\otimes_\Zl\Ql$ of  $E$. 
Then  $V$  is of type (W) of Weil weight 1 at any  $v\nmid\ell$,  
while it is of Hodge-Tate type of Hodge-Tate weight  0  or  1  at  $\l$.

If we do not assume the equality (G), 
we can in fact prove an equality which is fairly close to (G) 
under a mild condition:

\begin{proposition}\label{prop:(G)}
Let  $V$  be a $\Ql$-linear representation of  $\GK$  and  
$q$  a prime number $\not=\ell$. 
Assume  $V$  is of type (W) at all places  $v|q$  and 
of Hodge-Tate type at all places  $u|\ell$. Then we have
$$
   \sum_{v|q}   [\Kv:\Qq]\Sigma(\Wv(V))\ =\ 
   2\sum_{u|\ell}[\Ku:\Ql]\Sigma(\HTu(V)).
$$
\end{proposition}

\begin{proof}
The induced representation  $\IndKQ(V)$  is a representation 
of  $\GQ$  which is 
of type (W) at  $q$  and 
of Hodge-Tate type at  $\ell$, and hence we have 
$$
   \Sigma(\Wq(\IndKQ(V)))\ =\ 2\Sigma(\HTl(\IndKQ(V)))
$$
by Lemma \ref{lem:Q}. 
We then observe that
\begin{align*}
    \Wq(\IndKQ(V))\ &=\ \coprod_{v|q}   [\Kv:\Qq]\Wv(V), \\
   \HTl(\IndKQ(V))\ &=\ \coprod_{u|\ell}[\Ku:\Ql]\HTl(V),
\end{align*}
where the multiple  $mX$  of a multi-set  $X$  
by a positive integer $m$  is defined in the obvious manner.
Indeed,  we have 
$$
   (\IndKQ(V))|_{\Gq}\ =\ \bigoplus_{v|q}\IndKvQq(V|_{\Gv})
$$
by Mackey's formula (\cite{Serre:RL}, Section 7.3, Proposition 22), and 
$$
   \Wq(\IndKvQq(V|_{\Gv}))\ =\ [\Kv:\Qq]\Wv(V|_{\Gv})
$$
by definition of the induced representation and 
by the invariance of the Weil weights by finite extensions of 
the base field. 
Similar equalities hold for  $u|\ell$  and  
$\IndKuQl(V|_{\Gu})$.
\end{proof}

\section{Proof of the theorems}\label{sec:proof}

We begin with a version of the gap principle:

\begin{lemma}\label{lem}
Let  
$E,n,v$  be as before, and let  
$w\in\R_{\geq 0}$  be given.
Then there exists a constant  
$C_1=C_1([E:\mathbb{Q}],n,\qv^w)>0$  such that, for any prime  
$\ell>C_1$  and for any 
$n$-dimensional $\El$-linear representations  $V,V'$  of  $\Gv$ 
which are of type (W), $E$-integral at  $v$  and such that  
$\Sigma(\Wv(V)), \Sigma(\Wv(V'))$ are in $[0,[\El:\Ql]\cdot w]$,  
the following (i) and (ii) hold:

\sn
(i) 
If  $V\congss V'\pmodl$  as $\Gv$-representations, then  
$\Wv(V)=\Wv(V')$.   

\sn
(ii)  
Assume further that  $V^\ss$  and  $(V')^\ss$  are unramified. 
If  $V\congss V'\pmodl$  as $\Gv$-representations, then  
$V\simeqss V'$  as $\Gv$-representations.   

\smallskip
The constant  $C_1$  can be taken explicitly to be
$$
   C_1:=\left(2{n\choose[n/2]}\qv^{w/2}\right)^{[E:\Q]/\fl}.
$$
\end{lemma}

We have also the following mod $\ell$ version of (ii) above, 
in which the constant is independent of  $[E:\Q]$:

\begin{lemma}\label{lem'}
Let  
$E,n,v$  be as before, and let  
$w\in\R_{\geq 0}$  be given.
Then there exists a constant  $\tilde{C}_1=\tilde{C}_1(n,\qv^w)>0$  
such that, for any prime  $\ell>C_1$  and for any 
$n$-dimensional $\El$-linear representations  $V,V'$  of  $\Gv$  
such that  $V^\ss,(V')^\ss$  are unramified and  
which are of type (W), $E$-integral at  $v$  and such that  
$\Sigma(\Wv(V)), \Sigma(\Wv(V'))$ are in $[0,[\El:\Ql]\cdot w]$,
the following holds: If  
$\det(T-\Frobv |V)\equiv 
 \det(T-\Frobv |V')\ \pmod{\ell\OE})$, then one has  
$V\simeqss V'$  as $\Gv$-representations. 

The constant  $\tilde{C}_1$  can be taken explicitly to be
$$
   \tilde{C}_1:=2{n\choose[n/2]}\qv^{w/2}.
$$
\end{lemma}

\begin{proof}
As the proofs are similar, we only give a proof of 
Lemma \ref{lem}. 
Choose a totally ramified extension  $\Kvp/\Kv$  over which  
$V$  and  $V'$  become semistable (cf.\ Lem.\ \ref{lem:elimination}).
Let  
$P (T)=\det(T-\Frobvp|\ \Vss)$  and 
$P'(T)=\det(T-\Frobvp|\ (V')^\ss)$  
be the characteristic polynomials (taken as $\El$-linear representations)
of the Frobenius  $\Frobvp$  at  $v'$      
acting on the semisimplifications  $\Vss$  and  $(V')^\ss$, respectively. 
By assumption, they have coefficients in  $\OE$. 
By assumption on the weights, 
for any embedding  $E\incl\C$,  
the terms of  $T^{n-i}$  have coefficients of absolute value 
$\leq{n\choose i}\qv^{w/2}$ 
Note that  $\Sigma(\Wv(V))$  is the sum of the Weil weights of  
$V$  as a $\Ql$-linear representation, and hence the sum of 
the Weil weights of the roots of  $P(T)$  is in  $[0,w]$). 
Set  
$C_1:=(2\max_{0\leq i\leq n}{n\choose i}\qv^{w/2})^{[E:\Q]/\fl}
     =(2{n\choose [n/2]}\qv^{w/2})^{[E:\Q]/\fl}$. 
Then if  $\ell>C_1$, we have
\begin{align*}
			   		     &V\    \equivss\ V'    \pmodl\quad\text{as $\Gv$-representations} \\
    \Longleftrightarrow\ &P(T)\ \equiv  \ P'(T) \pmodl           \\
    \Longleftrightarrow\ &P(T)\ =       \ P'(T).               
\end{align*}
Here, the last equivalence follows from the next lemma.
This implies that  $\Wv(V)=\Wv(V')$. 
If  $V^\ss$  and  $(V')^\ss$  are unramified, then 
they are determined by the actions of  $\Frobv$, and hence 
the equality  $P(T)=P'(T)$  is equivalent to  
$V\simeqss V'$. 
\end{proof}

\begin{lemma}
Let  
$a$  be a non-zero integer of  $E$  and 
$C_0$  a real number  $>0$. 
If  $a\equiv 0\pmod\l$  $($resp.\ $a\equiv 0\ (\mathrm{mod}\ \ell \mathcal{O}_E))$
and  
$|\iota(a)|\leq C_0$  for any embedding  $\iota:E\incl\C$, 
then we have  
$\ell\leq C_0^{[E:\Q]/\fl}$
$($resp.\ $\ell\leq C_0)$.
\end{lemma}

\begin{proof}
If  $\l|a$ (resp.\ $\ell|a$ in $\mathcal{O}_E$), 
then by taking the norm  $N:E^\times\to\Q^\times$, 
we have  
$\ell^\fl\leq |N(a)|$
(resp.\ $\ell^{[E:\mathbb{Q}]}\leq |N(a)|$). 
If  $|\iota(a)|\leq C_0$, then by taking the norm 
(or product over all  $\iota$), we have  
$|N(a)|\leq C_0^{[E:\Q]}$. 
The required inequality follows from these two inequalities.
\end{proof}

We need one more lemma:

\begin{lemma}\label{lem:modell}
Let  $G$  be a profinite group and  
$T,T'$  be free  $\OEl$-modules on which  
$G$  acts continuously and $\OEl$-linearly. Let  
$(T/\l T)^\ss$  and  
$(T/\ell T)^\ss$  be the semisimplifications of  
$T/\l T$  and  
$T/\ell T$  as  $\kl[G]$-modules, respectively. 
Let  $e$  be the ramification index of  $\El/\Ql$. 
Then we have:

\sn
(i) 
$(T/\ell T)^\ss$  is isomorphic to the direct-sum of  $e$  copies of  
$(T/\l T)^\ss$. 

\sn
(ii)
If  
$(T/\l T)^\ss\simeq(T'/\l T')^\ss$, then  
$(T/\ell T)^\ss\simeq(T'/\ell T')^\ss$.
\end{lemma}

\begin{proof}
Part (ii) follows from Part (i) immediately. 
To prove (i), consider the filtration 
$$
    T/\ell T=T/\l^e T\ \supset\ \l T/\l^eT\ \supset\ \cdots
                     \ \supset\ \l^eT/\l^eT=0.
$$
Then 
``multiplication by $\l$'' 
(where  $\l$  is identified with a uniformizer at  $\l$)
induces isomorphisms 
$\l^i    T/\l^{i+1}T\to
 \l^{i+1}T/\l^{i+2}T$  
of the graded quotients as  $\kl[G]$-modules. 
It then follows that  
$(T/\ell T)^\ss\simeq((T/\l T)^\ss)^{\oplus e}$. 
\end{proof}

Now we can prove the theorems. 
We only prove Theorem \ref{thm:main} and \ref{thm:main2}, 
the proof of Theorem \ref{thm:main3} being similar.  
Let 
$C\:=\ \max\{e^2b+1,(2{n\choose[n/2]}\qv^{nb})^{[E:\Q]/\fl}\}$, 
as in Theorem \ref{thm:main}. 
Choose a finite totally ramified extension $K'_{u'}/K_u$, with absolute ramification
index $e^2$,
over which $V$ and $V'$ become semistable
(cf.\ Lem.\ \ref{lem:elimination}).
If  $\ell>C$, then  $e^2b<\ell-1$. 
Take $K'$ a finite extension of $K$ and $u'|u$ a place of $K'$ such that
the completion of $K'$ at $u'$ is $K'_{u'}$.
By assumption, we have  
$\HT_{u'}(V)\subset[0,b]$. 
Then by (i) of the Proposition \ref{prop:tame inertia}, we have 
$\TI_{u'}(V)\subset[0,b]$. The same holds for  $V'$, since we have
$\TI_{u'}(V)=\TI_{u'}(V')$   by the assumption  
$V\equivss V'\pmodl$  as $\Gu$-representations 
(Note that, by Lemma \ref{lem:modell}, we have also  
$V\equivss V'\pmod\ell$  as $\Fl[\Gu]$-modules, 
where  $V$  and  $V'$  are now regarded as $\Ql$-linear representations, 
so that the definition of  $\TIu$  and Proposition \ref{prop:tame inertia} are applicable).
Now we recall that $V$ and $V'$ are of type (G).
By (ii) of Proposition \ref{prop:tame inertia}, 
we have  
$\Sigma(\TI_{u'}(V)) 
=\Sigma(\HT_{u'}(V)) 
=\Sigma(\HTu(V)) 
=(1/2)\Sigma(\Wv(V))$, and these are also equal to 
$\Sigma(\TI_{u'}(V')) 
=\Sigma(\HT_{u'}(V')) 
=\Sigma(\HT_u(V')) 
=(1/2)\Sigma(\Wv(V'))$. 
Since  $\HTu(V)\subset[0,b]$, 
these are bounded by  $[\El:\Ql]\cdot nb$. 
In particular, total weights  
$\Sigma(\Wv(V))$  and  $\Sigma(\Wv(V'))$  
are  $\leq [\El:\Ql]\cdot 2nb$.
By (i) (resp.\ (ii)) of Lemma \ref{lem}, the assumption that  
$V\equivss V'\pmodl$  as $\Gv$-representations implies that   
$\Wv(V)=\Wv(V')$  (resp.\ 
$V\simeqss V'$  as $\Gv$-representations) if  
$\ell>(2{n\choose[n/2]}\qv^{nb})^{[E:\Q]/\fl}$.
\qed

\section{Corollaries}\label{sect:cor}

Here we give several corollaries of Theorem \ref{thm:main2}, 
which are motivated by a conjecture of Rasmussen and Tamagawa 
(\cite{Ras-Tama}; see also \cite{Bourdon}, \cite{Ozeki}, \cite{Ozeki2} and 
\cite{Ras-Tama2}). 
The notations ($K$, $E$, $n$, $b$, $e$, $v$, $u$, $\ell,\lambda$, 
$C=C([E:\Q],n,b,e,\qv)$, ...) 
are the same as in the theorem. 
In this section, 
$V=\VXr$  will be the $\El$-linear dual  $\Hetr(\XKbar,\El)^*$
of the $r$-th $\lambda$-adic \'etale cohomology group, where  
$X$  is a smooth proper variety 
(variety := separated scheme of finite type over a field) 
over  $K$  and  $\XKbar$  denotes its base extension to  $\Kbar$. 
We set  $\Vbar=\VbarXr:=T/\lambda T$, 
choosing a $\GK$-stable $\OEl$-lattice in  $V$, and let  
$\Vbarss=\VbarssXr$  be its semisimplification as a $\kl[\GK]$-module  
($\VbarssXr$  does not depend on the choice of  $T$).  
To state the first corollary, 
we make the following hypothesis on  $\Vbarss$:

\mn
{\it Hypothesis} (H). 
Each simple factor  
$\Wbar$  of  $\Vbarss$  lifts to an $\El$-linear representation    
$W$  of  $\GK$  of the form  
$\Hets(\YKbar,\El)^*$  
which is semistable 
at all  $u\mid\ell$  and  $\HTu(W)\subset[0,\ell-2]$, 
where  $Y$  is a proper smooth variety over  $K$  
and  $s$  is some non-negative integer.

\begin{corollary}\label{cor:H}
For any prime  $\ell>C$, 
any {\rm odd} integer  $r$  with  $1\leq r\leq b$,   
any places  $u$  of  $K$  and  $\l$  of  $E$  
both lying above  $\ell$, and any 
smooth proper variety  $X$  which has  
the $r$-th Betti number  $\leq n$, 
has potentially good reduction at  $v$, 
and has semistable reduction at some place  $u\mid\ell$, 
if (H) is true for  $\VbarssXr$,  
then none of the simple factors of  $\VbarssXr$  are of odd dimension.
\end{corollary} 

\begin{proof}
Note first that, if  $s$  is odd, then  
$\Hets(\YKbar,\El)$  has even dimension by (GAGA and) 
Hodge theory. Now, let  
$\Wbar_1,...,\Wbar_k$  be the simple factors of  
$\Vbarss$. 
By (H), each  $\Wbar_i$  lifts to a geometric  $W_i$  
with  $\HTu(W_i)\subset[0,\ell-2]$. 
If one of the  $W_i$  has odd dimension, then 
it must have even weight, 
while  $V$  has odd weight  $r$, since  
$X$  has potentially good reduction at  $v$. 
Thus the corollary follows from 
Theorem \ref{thm:main2} by putting  
$V':=W_1\oplus\cdots\oplus W_k$.
\end{proof}

As a special case where the Hypothesis (H) holds, 
we have:

\begin{corollary}
\label{cor:cycl} 
For any prime number  $\ell>C$, 
any {\rm odd} integer  $r$  with  $1\leq r\leq b$,  
any places  $u$  of  $K$  and  $\l$  of  $E$  
both lying above  $\ell$, and any smooth proper variety  
$X$  over  $K$  which has $r$-th Betti number  $\leq n$,
has potentially good reduction at  $v$, and 
has semistable reduction at  $u$, 
the Galois representation on  $\VbarssXr$  
is not the sum of integral powers mod $\ell$ cyclotomic characters.
\end{corollary}

In fact, we can generalize this a bit as follows. 
Let  $\chi$  and  $\chibar$  denote respectively the 
$\ell$-adic and mod $\ell$ cyclotomic characters of  $\GK$. 

\begin{corollary}
\label{cor:diagonal}
Assume  $E$  contains the $e^2$-th roots of unity. Then 
for any prime number  $\ell>C$  
such that  $\ell\equiv 1\ (\mathrm{mod}\ e^2)$, 
any {\rm odd} integer  $r$  with  $1\leq r\leq b$,  
any places  $u$  of  $K$  and  $\l$  of  $E$  
both lying above  $\ell$, and any smooth proper variety  
$X$  over  $K$  which has $r$-th Betti number  $\leq n$, 
has potentially good reduction at  $v$, and  
acquires semistable reduction over a finite extension  
$\Kup/\Ku$  with absolute ramification index  $e(\Kup/\Ql)$  dividing  $e$, 
the Galois representation  $\VbarssXr$  
is not the sum of characters of  $\GK$  of the form  
$\ebari\chibar^\bi$, where  $\ebari:\GK\to\kl^\times$  are 
characters unramified at  $u$  and of finite order dividing 
the order of the group of roots of unity in  $E$, 
and  $\bi$  are integers.
\end{corollary}

\begin{proof}
Suppose  $X$  has semistable reduction over  $\Kup$  
with  $e(\Kup/\Ql)\mid e$. We may assume  
$e(\Kup/\mathbb{Q}_{\ell})=e$.
Suppose  $\Vbarss$  is the sum of the characters    
$\ebari\chibar^\bi$  as above. 
Then the action of the tame inertia group  
$\Iuptame$  at  $u'$  on the $i$-th factor 
is via  $\chibar^\bi$, which equals  
$\theta^{e\bi}$, where  $\theta$  is the 
fundamental character of  $\Iuptame$  of level 1 
(\cite{Serre:PG}, Sect.\ 1.8, Prop.\ 8). 
By (i) of Proposition \ref{prop:tame inertia}, we have  
$e\bi\equiv\ci\pmod{\ell-1}$  with  
$0\leq\ci\leq eb$. Since  
$e^2\mid\ell-1$, 
we have  
$\bi=\boi+\frac{\ell-1}{e^2}j$
with   
$0\leq\boi\leq b$  and  $0\leq j<e^2$. 
Set  
$\kappabar:=\chibar^{(\ell-1)/e^2}$
and let  
$\kappa:\GK\to\El^\times$  be its Teichm\"uller lift. 
Since the $e^2$-th power of $\kappa$ is trivial,
it takes values in  $E^\times$.  
Similarly, the Teichm\"uller lift  $\ei$  of  $\eibar$  
has also values in  $E^\times$. 
Now each character  
$\eibar\chibar^\bi
=\eibar\kappabar^j\chibar^\boi$  lifts to the character  
$\ei   \kappa   ^j\chi   ^\boi:\GK\to\El^\times$, 
or to the 1-dimensional $\El$-linear $E$-integral geometric representation  
$\El(\ei\kappa^j)\otimes_\Ql\Ql(\boi)$, where  
$\El(\ei\kappa^j)$  is the twist of the trivial representation  $\El$  
by the finite character  $\ei\kappa^j$  and  
$\Ql(\boi)$  denotes the $\boi$-th Tate twist. 
Let  $V'$  be the direct-sum of these representations. 
By Theorem \ref{thm:main2}, we have  
$\Wv(V)=\Wv(V')$, but 
$\Wv(V)=\{r,...,r\}$  (since  $X$  has potentially good reduction at  $v$) while 
$\Wv(V')=\{2b_{01},...,2b_{0n}\}$, 
which is a contradiction if  $r$  is odd. 
\end{proof}

Specializing further, we have:

\begin{corollary}
Let  $K=\Q$. 
Assume  $E$  contains the $e^2$-th roots of unity. Then 
for any prime number  $\ell>C$  
such that  $\ell\equiv 1\ (\mathrm{mod}\ e^2)$, 
for any {\rm odd} integer  $r$  with  $1\leq r\leq b$,  and 
for any smooth proper variety  $X$  over  $\Q$   
which has  $r$-th Betti number  $\leq n$,  
has good reduction outside  $\ell$  and 
acquires semistable reduction over a finite extension  
$\Kup/\Ql$  with absolute ramification index  $e(\Kup/\Ql)$  dividing  $e$,
the Galois representation on  $\Vbar$
is not Borel. 
\end{corollary}

Here, we say that the representation  $\Vbar$  is Borel 
if the action of  $\GQ$  is given by upper-triangular matrices 
with respect to a suitable $\kl$-basis of  $\Vbar$. 

\begin{proof}
Indeed, if it is Borel, its semisimplification is a sum of 
characters, which are unramified outside  $\ell$  by assumption. 
Since the base field is  $\Q$, 
they are powers of the mod $\ell$ cyclotomic character. 
Now the the result follows from the previous corollary.
\end{proof}

\section{Congruences of modular forms}

We use the same notations as in the Introduction,
except that we always suppose  $K=\Q$  and write  
$q$  for  $q_v$  in this section. 
We put  $\varphi(N)=\#(\Z/N\Z)^{\times}$
for any positive integer  $N$
and denote by  $\Zbar$  the integer ring of  $\Qbar$.
The goal of this section is to give a proof of the following congruence 
result on the Fourier coefficients of modular forms. 
For any integers  $k,N\geq 1$  and a character 
$\epsilon:(\Z/N\Z)^\times\to\C^\times$, 
let  $\SkNe$  denote the $\C$-vector space of 
cusp forms of weight  $k$, level  $N$  and Nebentypus character  $\epsilon$. 
For a normalized Hecke eigenform  
$f(z)=\sum_{n=1}^\infty a_n(f)e^{2\pi inz}\in\SkNe$, 
integers  $i,j$  and a prime number  $\ell$,  
consider the following condition on the Fourier coefficients  
$a_p(f)$  of  $f$: 
\begin{equation}\label{eqn:Cijl}
\tag{$\mathrm{C}_{i,j:\ell}$}
   a_p(f)\equiv p^i+p^j\ (\text{mod}\ {\ell\Zbar})\quad
   \text{for all but finitely many primes  $p\nmid \ell N$}.
\end{equation}
For fixed  $k$  and  $N$, 
it is well known (cf.\ e.g.\ 
Thm.\ 10 of \cite{Serre:Cong} and 
the Introduction of \cite{Kiming-Verrill})
that there are only finitely many exceptional primes, 
and a fortiori finitely many primes  $\ell$  for which  
{\rm (\ref{eqn:Cijl})} hold for some  $i,j$  and  $f\in\SkNe$.  
Until recently, however, the situation had not been very clear when 
we let  $k$  and  $N$  vary; as for recent works, see 
\cite{Ghate-Parent} for the case of modular Abelian varieties and 
\cite {Billerey-Dieulefait} for the case of modular forms on  $\Gamma_0(N)$. 
In this vein, we show the following 
by using Theorem \ref{thm:main3}:

\begin{theorem}
\label{thm:modular} 
Fix a prime number  $q$. 
For any integer  $k\ge 1$, 
any prime  $\ell>4q^{2(k-1)}$,
any integer  $N$  such that  $q\nmid N$,
$\ell\nmid \varphi(N)$  and  $\ell^2\nmid N$, 
any character  $\epsilon:(\Z/N\Z)^\times\to\C^\times$, and 
any normalized Hecke eigenform  $f\in\SkNe$, we have the following:

\sn
(i)  
The condition {\rm (\ref{eqn:Cijl})} can hold only if   
$i\equiv j\equiv (k-1)/2\ \pmod{\ell-1}$.

\sn
(ii) 
The condition {\rm (\ref{eqn:Cijl})} holds for no  $i$  and  $j$  
if either 
$k=1$, 
$k$  is even, or 
$\ell\nmid N$. 
\end{theorem}

We begin by proving a lemma. 
For any  $f$  as in the theorem, 
we denote by  $E=\Q_f$  the field obtained by adjoining 
all Fourier coefficients of $f$ to $\Q$, 
which is a finite extension of $\Q$.
We regard  $\epsilon$  as a character with values in  $\OE^\times$.
Denote by 
$\ebar$ (resp.\ $\ebarl$) the composite 
$(\Z/N\Z)^\times\overset{\epsilon}{\to} \OE^\times
\overset{\mathrm{mod}\ \ell}{\to} (\OE/\ell \OE)^\times$
(resp.\ 
$(\Z/N\Z)^\times\overset{\epsilon}{\to} \OE^\times
\overset{\mathrm{mod}\ \lambda}{\to} (\OE/\lambda\OE)^\times$). Let  
$$
   \rfl:\ \GQ\ \to\ \GL_\El(\Vfl)
$$
be the 2-dimensional $\El$-linear representation of $\GQ$
associated with  $f$.
Thus if  $p\nmid\ell N$, then  $\Vfl$  is unramified at  $p$  and one has 
$$
   \det(T-\Frob_p|\Vfl)\ =\ T^2-a_p(f)T+\epsilon(p)p^{k-1}.
$$
In particular, 
it is $E$-integral at  $p$  in the sense of Definition \ref{def:E-int}.     
One has  $\W_p(\Vfl)=\{(k-1)/2,(k-1)/2\}$.  
It is  crystalline (resp.\ semistable) at  $\ell$  
if  $\ell\nmid N$ (resp.\ $\ell^2\nmid N$).

\begin{lemma}
\label{lem:modular}
Suppose  $\ell>2$.
Let  $k\ge 1$  and  $N\ge 1$  be integers with  $\ell\nmid\varphi(N)$.
Let  $\epsilon\colon(\Z/N\Z)^\times\to\C^\times$
be a character.
Suppose that a normalized Hecke eigenform 
$f\in\SkNe$  satisfies the condition {\rm (\ref{eqn:Cijl})} for some  $i,j$.
Then  $\ebar$  has values in fact in 
the canonical image of  $\Fl^\times$  in  $(\OE/\ell \OE)^\times$. 
Moreover, the following holds:

\sn
(i) 
We have  $\ebar(x\ (\mathrm{mod}\ N))=x^{i+j-(k-1)}\ (\mathrm{mod}\ \ell)$
for any  $x$  prime to  $N$.

\sn
(ii) 
If  $\ell\nmid N$, 
then we have  $i+j\equiv k-1 \pmod{\ell-1}$  and  $\bar{\epsilon}=1$.
\end{lemma}

\begin{proof}
By assumption, we have 
$\Tr(\Frob_p|\Vfl)\equiv p^i+p^j\ (\mathrm{mod}\ \ell\OE)$
for all but finitely many  $p\nmid \ell N$.
In particular,
we have  
\begin{equation}\label{eqn:cong}
   \rfl\ \congss\ \chi^i\oplus\chi^j \pmod{\lambda}
\end{equation}
as $\kl$-linear representations of  $\GQ$
(This holds because  $\ell>\dim\rfl$; 
see e.g.\ Lemma 2.10 of \cite{Ozeki2}),
and then we have also 
$\epsilon(p)p^{k-1}\equiv p^{i+j}\ (\mathrm{mod}\ \lambda)$.
Hence we see that 
\begin{equation}\label{eqn:epsilon}
   \ebarl(x\ (\mathrm{mod}\ N))\ =\ x^{i+j-(k-1)}\ \pmod\lambda
\end{equation}
for any  $\lambda|\ell$  and any integer  $x$  prime to  $N$.

(i) 
Since the kernel of the projection 
$(\OE/\ell\OE)^\times\to 
\prod_{\lambda\mid\ell} (\OE/\lambda\OE)^\times$
has $\ell$-power order,  
if  $\ell\nmid\varphi(N)$, then the homomorphism  
$\prod_{\lambda\mid \ell}\bar{\epsilon}_{\lambda}:
(\Z/N\Z)^\times\to\prod_{\lambda\mid\ell} (\OE/\lambda\OE)^\times$  
lifts uniquely to a homomorphism  
$(\Z/N\Z)^\times\to(\OE/\ell\OE)^\times$, which is   
$\bar{\epsilon}$. 
According to (\ref{eqn:epsilon}), it is given by 
\begin{equation}\label{eqn:epsilon2}
   \ebar(x\ (\mathrm{mod}\ N))\ =\ x^{i+j-(k-1)}\pmod{\ell\OE}
\end{equation}
for any integer  $x$  prime to  $N$.

(ii) 
Suppose  $\ell\nmid N$. 
Then (\ref{eqn:epsilon2}) must hold for  $x=\ell$, 
which is possible only if  $i+j\equiv k-1$ (mod $\ell-1$). 
In particular, we obtain  $\ebar=1$. 
\end{proof}

\begin{proof}[Proof of Theorem \ref{thm:modular}]
(i) 
Suppose  
$\ell\nmid\varphi(N)$  and  $\ell^2\nmid N$. 
Then  $\rfl$  is semistable at  $\ell$.  
By assumption, we have 
$\Tr(\Frob_q|\Vfl)\equiv q^i+q^j\ (\mathrm{mod}\ \ell\OE)$.
Combining this with Lemma \ref{lem:modular} (i), we obtain  
$\det(T-\Frob_q|\Vfl)\equiv 
 \det(T-\Frob_q|\chi^i\oplus \chi^j)$ (mod $\ell\OE$).
We also have the congruence (\ref{eqn:cong}). 
Therefore, if  $\ell>4q^{2(k-1)}$, 
it follows from 
Theorem \ref{thm:main3} 
(applied with  
$V'=\chi^{i'}\oplus\chi^{j'}$, where  
$i',j'$  are integers in  $[0,\ell-2]$  such that  
$i'\equiv i$, 
$j'\equiv j$ (mod $\ell-1$)) that 
$\rfl\simeqss\chi^i\oplus\chi^j$  
as $\El$-linear representations of the decomposition group  $G_q$  of  $q$.
Looking at the Weil weights, we obtain  
$i\equiv j\equiv (k-1)/2\ (\mathrm{mod}\ \ell-1)$. 

(ii)
If  $k$  is even, then the impossibility of (\ref{eqn:Cijl}) 
follows from Part (i). 

If  $k=1$  and the congruence condition (\ref{eqn:Cijl}) holds, 
then Part (i) together with (\ref{eqn:cong}) implies that 
$\rflbar:=\rfl$ (mod $\lambda$) is unipotent and, 
in particular, $\Im(\rflbar)$  is an $\ell$-group. 
On the other hand, 
if  $k=1$, then by \cite{Deligne-Serre}, 
$\Im(\rfl)$  is finite and its image in  
$\mathrm{PGL}_2(\OEl)$  is either 
dihedral, $A_4$, $S_4$  or  $A_5$. 
Since the kernel of the reduction map  
$\GL_2(\OEl)\to\GL_2(\kl)$  is pro-$\ell$, the representation  
$\rflbar$  cannot be unipotent if  $\ell\geq 3$. 

Finally, assume  $\ell\nmid N$. 
Then  $\rfl$  is crystalline at  $\ell$,
and thus the Fontaine-Laffaille theory \cite{Fontaine-Laffaille} 
implies that 
the tame inertia weights and the Hodge-Tate weights of  $\rfl$ 
coincide with each other. Hence it follows from 
(\ref{eqn:cong}) that
$\{i,j \}\equiv \{0,k-1\}\ (\mathrm{mod}\ \ell-1)$.
Since  $\ell>k$, we obtain  $\{(k-1)/2,(k-1)/2\}=\{0,k-1\}$, 
which is impossible unless $k=1$.
\end{proof}


\medskip\noindent
(Y.\ O.) 
Research Institute for Mathematical Sciences, Kyoto University \\ 
Kyoto, 606-8502 Japan \\
Email address: {\tt yozeki@kurims.kyoto-u.ac.jp}

\sn
(Y.\ T.)
Faculty of Mathematics, Kyushu University \\
744, Motooka, Nishi-ku, Fukuoka, 819-0395 Japan\\
Email address: {\tt taguchi@math.kyushu-u.ac.jp}

\end{document}